\def\FormatStyle{Journal}  
\def\InHouse{InHouse}
\def\Conference{Conference}

\newif\ifblind
\blindfalse 

\ifx \FormatStyle \InHouse  
  \documentclass[a4paper,10pt,onecolumn]{article}
  \usepackage[margin=1in]{geometry}
  \usepackage{setspace}
  \usepackage{fancyhdr}
\else 
  \ifx\FormatStyle\Conference  
  \documentclass[a4paper,10pt,twocolumn]{article}
  \else 
  \documentclass[a4paper,10pt,onecolumn]{article}
  \fi
  \usepackage[margin=1in]{geometry}
\fi

\usepackage[utf8]{inputenc}
\usepackage[noblocks,affil-it]{authblk}
\usepackage[cmex10]{amsmath}
\usepackage{amsthm,amsfonts,amssymb}
\usepackage[verbose]{wrapfig}
\usepackage{IEEEtrantools}  
\usepackage[pdftex]{graphicx}
\usepackage[usenames,dvipsnames,svgnames,table]{xcolor}
\graphicspath{{./Figures/}}
\usepackage{appendix}
\usepackage{xspace}

\usepackage{algorithm,algpseudocode}
\algrenewcommand{\algorithmiccomment}[1]{\hskip3em - #1}

\ifx\FormatStyle\InHouse
 \usepackage[disable]{todonotes}
\else
 \usepackage[disable]{todonotes}
\fi

\usepackage{natbib} 
\usepackage{hyperref} 
\definecolor{Pcolor}{rgb}{0.20,0.50,0.40}
\hypersetup{colorlinks=true, citecolor = Pcolor}

\newtheorem{thm}{Theorem}

\newtheorem{lemma}{Lemma}

\newcommand{\mynotesH}[2][]{\todo[backgroundcolor=blue!20!white,inline,
bordercolor=red]{#2}}
\newcommand{\mynotesL}[2][]{\todo[backgroundcolor=blue!20!white,inline,
bordercolor=yellow]{#2}}
\newcommand{\mynotesS}[2][]{\todo[backgroundcolor=blue!20!white,
bordercolor=orange]{#2}}

\newcommand{\E}{\mathbb{E}}
\newcommand{\Prob}{\mathbb{P}}

\newcommand{\ud}{\,\mathrm{d}}

\newcommand{\I}{\mathbf{I}}
\newcommand{\Hm}{\mathbf{H}}

\newcommand{\X}{\mathbf{X}}
\newcommand{\x}{\mathbf{x}}

\newcommand{\y}{\mathbf{y}}

\newcommand{\veps}{\varepsilon}
\newcommand{\bveps}{\boldsymbol{\varepsilon}}
\newcommand{\bbeta}{\boldsymbol{\beta}}
\newcommand{\pd}[2]{\frac{\partial #1}{\partial #2}}
\newcommand{\torp}[2]{\texorpdfstring{#1}{#2}}

\ifblind
  \author{}
\else
  \author[1]{Kory D. Johnson}
  \author[2]{Dongyu Lin}
  \author[3]{Lyle H. Ungar}
  \author[1]{Dean P. Foster}
  \author[1]{Robert A. Stine}
  \affil[1]{The Wharton School\\University of Pennsylvania}
  \affil[2]{AT\&T Labs}
  \affil[3]{The School of Engineering and Applied Science\\University of
Pennsylvania}
\fi

\ifx \FormatStyle \InHouse
 \title{Draft: A Risk Ratio Comparison of $l_0$ and $l_1$ Penalized Regression}
\else
 \title{A Risk Ratio Comparison of $l_0$ and $l_1$ Penalized Regression}
\fi

\begin{document}

\maketitle

\listoftodos

\begin{abstract}
There has been an explosion of interest in using $l_1$-regularization in place
of $l_0$-regularization for feature selection. We present theoretical results
showing that while $l_1$-penalized linear regression never outperforms
$l_0$-regularization by more than a constant factor, in some cases using an
$l_1$ penalty is infinitely worse than using an $l_0$ penalty. We also show that
the ``optimal''  $l_1$ solutions are often inferior to $l_0$ solutions found
using stepwise regression.

We also compare algorithms for solving these two problems and show that although
solutions can be found efficiently for the $l_1$ problem, the ``optimal''  $l_1$
solutions are often inferior to $l_0$ solutions found using greedy classic
stepwise regression. Furthermore, we show that solutions obtained by solving the
convex $l_1$ problem can be improved by selecting the best of the $l_1$ models
(for different regularization penalties) by using an $l_0$ criterion.
In other words, an approximate solution to the right problem can be better than
the exact solution to the wrong problem.

\end{abstract}



\section{Introduction}

In the past decade, a rich literature has been developed using
$l_1$-regularization for linear regression including Lasso \citep{Tib96}, LARS
\citep{Efron+04}, fused lasso \citep{Tib+05}, elastic net \citep{ZouH05},
grouped lasso \citep{YuanL06}, adaptive lasso \citep{Zou06}, and relaxed lasso
\citep{Mein07}. These methods, like the $l_0$-penalized regression methods which
preceded them \citep{Aka74,Schwarz78,FosterG94}, address variable selection
problems in which there is a large set of potential features, only a few of
which are likely to be helpful.  This type of sparsity is common in machine
learning tasks, such as predicting disease based on thousands of genes, or
predicting the topic of a document based on the occurrences of hundreds of
thousands of words.

$l_1$-regularization is popular because, unlike the $l_0$ regularization
historically used for feature selection in regression problems, the $l_1$
penalty gives rise to a convex problem that can be solved efficiently using
convex optimization methods. $l_1$ methods have given reasonable results on a
number of data sets, but there has been no careful analysis of how they perform
when compared to $l_0$ methods.  This paper provides a formal analysis of the
two methods, and shows that $l_1$ can give arbitrarily worse models. We offer
some intuition as to why this is the case -- $l_1$ shrinks coefficients too much
and does not zero out enough of them -- and suggest how to use an $l_0$ penalty
with $l_1$ optimization.

We study the problem of selecting predictive features from a large feature
space.  We assume the classic normal linear model
\begin{IEEEeqnarray*}{rCl+rCl}
 \y & = & \X\beta + \epsilon & \epsilon & \sim & N_n(\mathbf{0},\sigma^2 \I_n)
\end{IEEEeqnarray*}
with $n$ observations $\y = (y_1,\ldots,y_n)'$ and $p$ features
$\x_1,\ldots,\x_p$, where $\X = (\x_1,\ldots,\x_p)$ is an $n \times
p$ ``design matrix'' of features, $\bbeta = (\beta_1,\ldots,\beta_p)'$ is the
coefficient parameters, and error $\bveps \sim N(\mathbf{0},\sigma^2 I_n)$. 

We expect most of the elements of $\beta$ to be 0. Hence, generating good
predictions requires identifying the small subset of predictive features. This
standard linear model proliferates the statistics and machine learning
literature. In modern applications, $p$ can approach millions, making the
selection of an appropriate subset of these features essential for prediction.
The size and scope of these problems raise concerns about both the speed and
statistical robustness of the selection procedure. Namely, it must be fast
enough to be computationally feasible and must find signal without over-fitting
the data.

The traditional statistical approach to this problem, namely, the
$l_0$ regularization problem, finds an estimator that minimizes the
$l_0$ penalized sum of squared errors,
\begin{equation}
\label{eqn:l0}
{\arg\min}_{\bbeta}\left\{ \|\y - \X\bbeta\|^2 + 
  \lambda_0 \|\bbeta\|_{l_0}\right \},
\end{equation}
where $\|\bbeta\|_{l_0} = \sum_{i=1}^p I_{\{\beta_i\neq0\}}$ counts the
number of nonzero coefficients.
However, this problem is NP-hard \citep{Nata95}. A tractable problem
relaxes the $l_0$ penalty to the $l_1$ norm,
$\|\bbeta\|_{l_1} = \sum_{i=1}^p|\beta_i|$, and seeks
\begin{equation}
\label{eqn:l1}
{\arg\min}_{\bbeta} \left\{\|\y - \X\bbeta \|^2 + 
  \lambda_1 \|\bbeta\|_{l_1} \right\},
\end{equation}
This is known as the $l_1$-regularization problem \citep{Tib96}, which is
convex. This problem can be solved efficiently using a variety of methods
\citep{Tib96,Efron+04,CanT07}.

We assess our models using the predictive risk function (\ref{eqn:risk})
\begin{IEEEeqnarray}{rCcCl}
\label{eqn:risk}
R(\bbeta, \hat\bbeta) & = & \E_{\bbeta} \| \hat \y - \E(\y|\X)\|_2^2
  & = & \E_{\bbeta} \|\X \hat\bbeta - \X\bbeta \|_2^2.
\end{IEEEeqnarray}
We are interested in the ratios of the risks of the estimates provided by these
two criteria. Unlike risk functions, predictive risk measures the relevant
prediction error on future observations, ignoring irreducible variance. Smaller
risks imply better expected prediction performance. It is an ideal metric to
analyze testing error or out-of-sample errors when the parameter distribution is
assumed to be known. Recent literature has focused on selection consistency:
whether or not the true variable can be identified in the limit. However, in
real application, due to ubiquitous multicollinearity, predictors are hard to
separate as ``true" and ``false". Here, we focus on predictive accuracy and
advocate the concept of predictive risk.

Our first result in this paper, given below as Theorems \ref{thm:low} and
\ref{thm:up} in Section \ref{sec:risk}, is that $l_0$ estimates provide more
accurate predictions than $l_1$ estimates do, in the sense of minimax risk
ratios. This is illustrated in Figure \ref{fig:supriskratio}. Proofs of these
theorems are in Appendix \ref{app:riskpf}.

\begin{itemize}
\item
$\inf_{\gamma_0} \sup_{\bbeta} \frac{R(\bbeta, \hat\bbeta_{l_0} (\gamma_0))}
  {R(\bbeta, \hat\bbeta_{l_1}(\gamma_1))}$
  is bounded by a small constant; furthermore, it is close to one
  for most $\gamma_1$s, especially for large $\gamma_1$s,
  which are mostly used in sparse systems.

\item 
$\inf_{\gamma_1} \sup_{\bbeta} \frac{R(\bbeta, \hat\bbeta_{l_1} (\gamma_1))}
  {R(\bbeta, \hat\bbeta_{l_0}(\gamma_0))}$
  tends to infinity quadratically; in an extremely
  sparse system, the $l_1$ estimate may perform arbitrarily badly.

\item $R(\bbeta, \hat\bbeta_{l_1}(\gamma_1))$ is more likely to have a
  larger risk than $R(\bbeta,\hat\bbeta_{l_0}(\gamma_0))$ does.
  \mynotesH{how is ``more likely'' defined?}
\end{itemize}

\begin{figure}[htbp]
\begin{center}
\centerline{\includegraphics[width=1.2\textwidth]{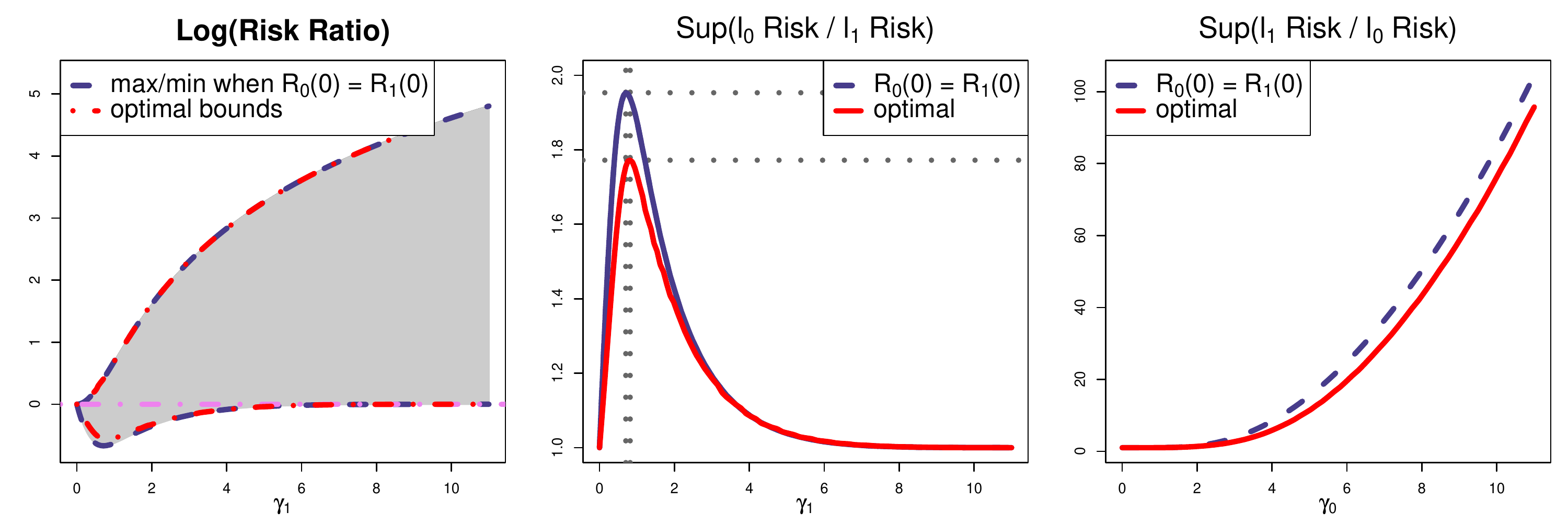}}
\vskip -.1in
\caption{\footnotesize{\textbf{Left}: The gray area shows the feasible
    region for the risk ratios--the log risk-ratio is above zero when
    $l_0$ produces a better fit.  The graph shows that most of the
    time $l_0$ is better.  The actual estimators being compared are
    those that have the same risk at $\beta=0$, i.e.,
    $R(0,\hat{\beta}_{l_0}(\gamma_0))
    =R(0,\hat{\beta}_{l_1}(\gamma_1))$.  
    \textbf{Middle}: This graph
    traces out the bottom envelope of the left hand graph (but takes
    the reciprocal risk ratio and no longer uses the logarithm
    scale). The dashed blue line displays
$\sup_{\beta}{R(\beta,\hat{\beta}_{l_0}(\gamma_0))}/{R(\beta,\hat{\beta}_{l_1}
(\gamma_1))}$
    for $\gamma_0$ calibrated to have the same risk at zero as
    $\gamma_1$. This maximum ratio tends to 1 when
    $\gamma_1\rightarrow0$ (the saturated case) or $\infty$ (the
    sparse case). With an optimal choice of $\gamma_0$,
$\inf_{\gamma_0}\sup_{\beta}{R(\beta,\hat{\beta}_{l_0}(\gamma_0))}/{R(\beta,\hat
{\beta}_{l_1}(\gamma_1))}$
    (solid red line) behaves similarly. Specifically, the supremum
    over $\gamma_1$ is bounded by 1.8.  
    \textbf{Right}: This graph
    traces out the upper envelopes of the left hand graph on a normal
    scale. When $\gamma_0\rightarrow\infty$,
$\sup_{\beta}{R(\beta,\hat{\beta}_{l_1}(\gamma_1))}/{R(\beta,\hat{\beta}_{l_0}
(\gamma_0))}$
    tends to $\infty$, for both $\gamma_1$ that is calibrated at
    $\beta=0$ and that minimizes the maximum risk ratio.  }}
\label{fig:supriskratio}
\end{center}
\vskip -0.1in
\end{figure} 

A detailed discussion on the risk ratios will be presented in Section
\ref{sec:risk}, along with a discussion of other advantages of $l_0$
regularization. Our other comparative results include showing that applying the
$l_0$ criterion on an $l_1$ subset searching path can find the best performing
model (Section \ref{sec:l0pen}) and running stepwise regression and Lasso on a
reduced NP hard example shows that stepwise regression gives better solutions
(Section \ref{sec:np}).

\mynotesH{Check all assumptions or remove.}
We compare $l_0$ vs. $l_1$ penalties under three assumptions about the structure
of the feature matrix $X$: independence, incoherence (near independence) and
NP-hard. For independence, we find provide the theoretical results mentioned
above. For near independence, we find that $l_1$ penalized regression followed
by $l_0$ outperforms $l_1$ selection. For the NP-hard case, we find that if one
could do the search, then the risk ratio could be arbitrarily bad for $l_1$
relative to $l_0$.

\section{Risk Ratio Comparison}
\label{sec:risk}

We assess our models using the predictive risk function (\ref{eqn:risk})
\begin{IEEEeqnarray}{rCcCl}
R(\bbeta,\hat\bbeta) & = & \E_{\bbeta} \| \hat\y - \E(\y|\X) \|_2^2 & = &
  \E_{\bbeta} \|\X \hat\bbeta - \X\bbeta \|_2^2.
\end{IEEEeqnarray}

This is the relevant component after decomposing the expected squared error loss
from predicting a new observation. This is clear from the following standard
decomposition. For increased generality, let $\E[y] = \eta]$ and $\Hm_\X$ be the
projection onto the column space of $\X$. Then
\begin{IEEEeqnarray*}{rCl}
\E\|y^* - \X \hat\beta\|^2 & = & \E\|y^*-\eta\|^2 + \E\|\eta - \Hm_\X\eta\|^2 +
  \E\|\Hm_\X \eta - \X \hat\beta\|^2 \\
& = & \underbrace{n \sigma^2}_{\text{common error}} + 
  \underbrace{\|(\I - \Hm_\X)\eta\|^2}_{\text{wrong $\X$}} +
  \underbrace{\E\|\Hm_\X \eta - \X \hat\beta\|^2}_{\text{predictive risk}}
\end{IEEEeqnarray*}

The first term, common error, is unavoidable, regardless of the method being
used. All methods we consider, namely linear methods based on $\X$, suffer
the error from incorrect $\X$. Since $\X$ is given, it is more instructive to
consider the projection of $\eta$ onto the column space of $\X$, defining
$\X\beta = \Hm_\X\eta$. Ignoring these two forms of error, leaves the predictive
risk function (\ref{eqn:risk}).
 
\mynotesH{ citations: \citet{BarbB04}; content: what is it, why is it good,
where has it been used; risk vs consistence; and relation to out-of-sample
error}

Predictive risk has guided selection procedures such as Mallow's $C_p$ and RIC.
The former results from an unbiased estimate of the predictive risk, while the
later provides minimax control of the risk in during model selection. We
maintain this minimax viewpoint and show that in terms of the \emph{removable}
variation in prediction, $l_0$ performs better than $l_1$.

%

\subsection{\torp{$l_0$}{l0} solutions give more accurate predictions.}
\label{sec:riskratio}

Suppose that $\hat\bbeta$ is an estimator of $\bbeta$. 
For this section, we assume $\X$ is orthogonal. (For example, wavelets, Fourier
transforms, and PCA all are orthogonal). The $l_0$ problem (\ref{eqn:l0}) can
then be solved by simply picking those predictors with least squares estimates
$|\hat\beta_i|>\gamma$, where the choice of $\gamma$ depends on the penalty
$\lambda_0$ in (\ref{eqn:l0}). It was shown \citep{DonJ94,FosterG94} that
$\lambda_0 = 2\sigma^2 \log p$ is optimal in the sense that it asymptotically
minimizes the maximum predictive risk inflation due to selection.

Let \begin{IEEEeqnarray}{rCl}
\label{eqn:l0hat}
\hat\bbeta_{l_0} (\gamma_0) & = & \left( \hat\beta_1 I_{\{|\hat\beta_1| >
  \gamma_0\}}, \ldots, \hat\beta_p I_{\{ |\hat\beta_p| > \gamma_0\}}\right)'
\end{IEEEeqnarray}
be the $l_0$ estimator that solves (\ref{eqn:l0}),
and let the $l_1$ solution to (\ref{eqn:l1}) be
\begin{IEEEeqnarray}{rCl}
\label{eqn:l1hat}
\hat\bbeta_{l_1} (\gamma_1) & = & \left( \mathrm{sign}(\hat\beta_1)
  (|\hat\beta_1|-\gamma_1)_+, \ldots, \mathrm{sign} (\hat\beta_p)
  (|\hat\beta_p|-\gamma_1)_+\right)',
\end{IEEEeqnarray}
where the $\hat\beta_i$'s are the least squares estimates. 

We are interested in the ratios of the risks of these two estimates, 
\begin{IEEEeqnarray*}{c'C'c}
\frac{R(\bbeta, \hat\bbeta_{l_0} (\gamma_0))}{R(\bbeta, \hat\bbeta_{l_1}
  (\gamma_1))} 
  & \text{and} &  
  \frac{R(\bbeta, \hat\bbeta_{l_1} (\gamma_1))}{R(\bbeta, \hat\bbeta_{l_0}
  (\gamma_0))}.
\end{IEEEeqnarray*}

I.e., we want to know how the risk is inflated when another criterion is used.
The smaller the risk ratio, the less risky (and hence better) the numerator
estimate is, compared to the denominator estimate. Specifically, a risk ratio
less than one implies that the top estimate is better than the bottom estimate.

Formally, we have the following theorems, whose proofs are given in Appendix A:
\begin{thm}
\label{thm:low}
There exists a constant $C_1$ such that for any $\gamma_0\geq 0$,
\begin{IEEEeqnarray*}{rCl}
\inf_{\gamma_1}\sup_{\bbeta}\frac{R(\bbeta,\hat{\bbeta}_{l_1}(\gamma_1))}{
R(\bbeta , \hat\bbeta_{l_0}(\gamma_0))}\geq C_1+\gamma_0.
\end{IEEEeqnarray*}
\end{thm}
I.e., given $\gamma_0$, for any $\gamma_1$, there exist $\bbeta$'s such that the
ratio becomes extremely large.

Contrast this with the protection provided by $l_0$:
\begin{thm}
\label{thm:up}
There exists a constant $C_2>0$ such that for any $\gamma_1\geq 0$, 
\begin{IEEEeqnarray*}{rCl}
\inf_{\gamma_0}\sup_{\bbeta}\frac{R(\bbeta,\hat{\bbeta}_{l_0}(\gamma_0))}{
  R(\bbeta , \hat{\bbeta}_{l_1}(\gamma_1))} & \leq & 1+C_2\gamma_1^{-1}.
\end{IEEEeqnarray*}
\end{thm}
I.e., for any $\gamma_1$, we can pick the $l_0$ cutoff so that we perform
almost as good as $l_1$, even in the worst case.

The above theorems can definitely be strengthened, as demonstrated by the bounds
shown in Figure \ref{fig:supriskratio}, but at the cost of complicating the
proofs. We conjecture that there exist constants $r>1$, and $C_3, C_4, C_5>0$,
such that
\begin{IEEEeqnarray}{rCl}
\inf_{\gamma_1} \sup_{\bbeta} \frac{R(\bbeta, \hat\bbeta_{l_1}
  (\gamma_1))}{R(\bbeta, \hat\bbeta_{l_0} (\gamma_0))} 
  & \geq & 1+C_3 \gamma_0^r,\\
\inf_{\gamma_0} \sup_{\bbeta} \frac{R(\bbeta, \hat\bbeta_{l_0}
  (\gamma_0))}{R(\bbeta, \hat\bbeta_{l_1} (\gamma_1))} 
  & \leq & 1+C_4\gamma_1 e^{-C_5\gamma_1}.
\end{IEEEeqnarray}

These theorems suggest that for any $\gamma_1$ chosen by the algorithm, we can
always adapt $\gamma_0$ such that $\hat\bbeta_{l_0} (\gamma_0)$ outperforms
$\hat\bbeta_{l_1} (\gamma_1)$ most of the time and loses out a little for some
$\bbeta$'s; but for any $\gamma_0$ chosen, no $\gamma_1$ can perform
consistently well on all $\bbeta$'s.

Because of the additivity of risk functions, (see appendix equations
(\ref{eqn:l0risk}) and (\ref{eqn:l1risk})), due to the orthogonality
assumption, we focus on the individual behavior of $\beta_i$ for each
single feature. Also the risk functions are symmetric on $\beta$, so
only the cases of $\beta_i \geq 0$ will be displayed. Figure \ref{fig:riskratio}
illustrates that given $\gamma_1$, we can pick a $\gamma_0$, s.t. the risk ratio
is below 1 for most $\beta$ except around $(\gamma_0+\gamma_1)/2$, yet this
ratio does not  exceed one by more than a small factor, even for the worst case.


\begin{figure}[htbp]
\begin{center}
\label{fig:riskratio}
\centerline{\includegraphics[width = 3 in]{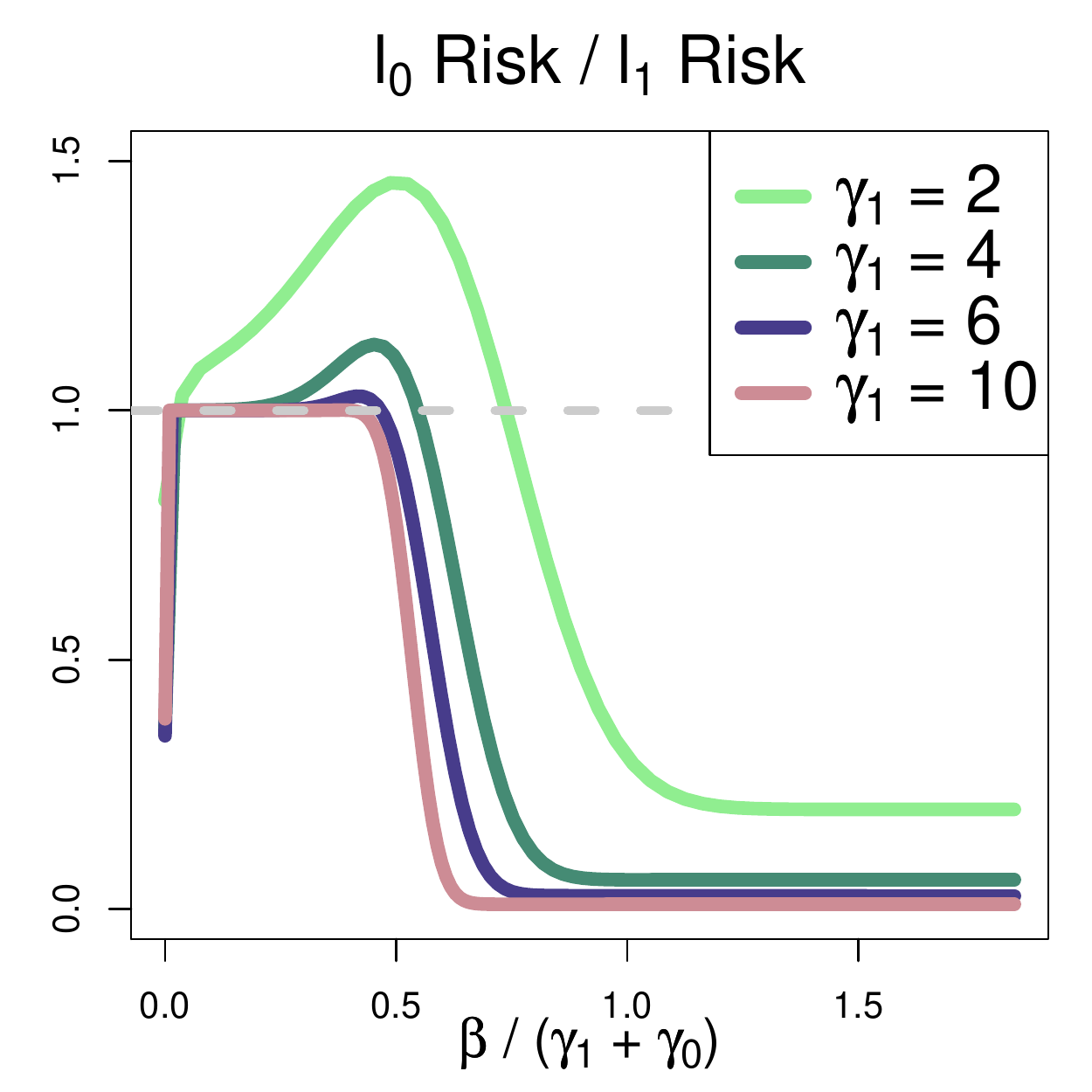}}
\vskip -.1in
\caption{{\footnotesize For each $\gamma_1$, we let
    $\gamma_0=\gamma_1+4\log(\gamma_1)/\gamma_1$. This choice of
    $\gamma_0$ makes the risk ratios small at $\beta \approx 0$ and
    $\beta\geq \gamma_0$, only inflated around
    $\beta/(\gamma_0+\gamma_1)=1/2$, albeit very little especially
    when $\gamma_1$ is large enough.}}
\end{center}
\vskip -0.1in
\end{figure} 


The intuition as to why $l_0$ fares better than $l_1$ in the risk ratio results
is that $l_1$ must make a ``devil's choice'' between shrinking the coefficients
too much or putting in too many spurious features.  $l_0$ penalized regression
avoids this problem.

\subsection{\torp{$l_1$}{l1} shrinks coefficients too much}
\label{sec:shrinkage}

\begin{figure}[htbp]
\begin{center}
\centerline{\includegraphics[width=1.1\textwidth]{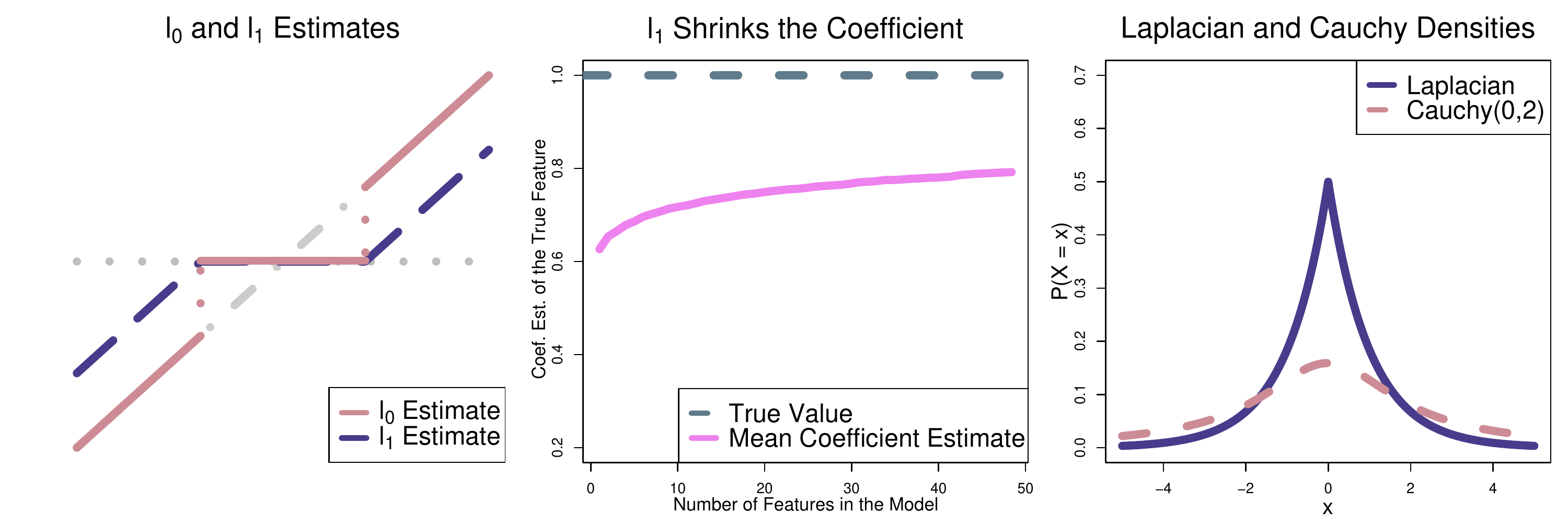}}
\vskip -.1in
\caption{{\footnotesize\textbf{Left}: The $l_0$ estimate keeps the
    least squares value after the cutting point, but the $l_1$ estimate
    always shrinks the least squares estimate by a fixed amount.
    \textbf{Middle}: the model we simulate has only one true feature
    with true $\beta=1$ and a thousand spurious features. We compute
    the average Lasso estimate of $\beta$ for a fixed number of
    features included in the model (as an index of the $l_1$ penalty)
    from several different trials. $\bar{\hat\beta}_{l_1}$ is always
    shrunk by at least 20\% in this experiment. \textbf{Right}: The Cauchy
    density has heavier tails than the Laplacian density does. Thus, a
    Laplacian prior tends to shrink large values of $\beta$'s.}}
\label{fig:shrinkage}
\end{center}
\vskip -0.1in
\end{figure} 

From a frequentist's point of view, the $l_1$ estimator (\ref{eqn:l1hat})
shrinks the coefficients and thus is biased (Figure \ref{fig:shrinkage}). In
practice, $\hat\beta_{l_1}$ can be substantially shrunk towards zero
when the system is sparse, as shown in the middle panel of Figure
\ref{fig:shrinkage}.

From a Bayesian's perspective, the $l_1$ penalty is equivalent to putting a
Laplacian prior on $\beta$ \citep{Tib96, Efron+04}, while the $l_0$ penalty can
be approximated by Cauchy priors \citep{JohnS05,FosterS05}.  The right panel of
Figure \ref{fig:shrinkage} shows that the Cauchy distribution has a much heavier
tail than the Laplacian distribution does. This implies that when the true
$\beta$ is far away from 0, the $l_1$ penalty will substantially shrink the
estimate toward zero.

The bias caused by the shrinkage increases the predictive risk proportionally to
the squared amount of the shrinkage. The sparser the problem is, the greater the
shrinkage is, thus the larger the risk is. These results show that in theory the
$l_0$ estimate has a lower risk and provides a more accurate solution.
Empirically, stepwise regression performs well in large data sets, where a
sparse solution is particularly preferred \citep{GeorgeF00,FosterS04,Zhou+06}.

\subsection{Simulations for Risk Ratio/flaws with l1}
\mynotesH{should mix some of the 1.2 and 1.3}
\subsubsection{\torp{$l_1$}{l1} optimization using an \torp{$l_0$}{l0}
criterion} 
\label{sec:l0pen}

We can make use of the LARS algorithm to generate a set of candidate solutions
and then use the $l_0$ criterion to find the best of the solutions along the
regularization path. We evaluated this method as follows.
We simulated $\mathbf{y}$ from a thousand features, only 4 of which have
nonzero contributions, plus a random noise distributed as
$N(0,1)$. Both the training set and the test set have size $n=100$. We
apply the Lasso algorithm implemented by LARS on this synthetic data
set. For each step on the regularization path, this algorithm selects
a subset $\mathcal{C}\subset\{1,\ldots,1000\}$ of features that are
included in the model. We then adopt a modified RIC criterion
suggested in \citet{GeorgeF00}:
\begin{equation}
\label{eqn:modric}
\| \y - \X_\mathcal{C} \hat\bbeta_\mathcal{C} \|_2^2 +
  \sum_{q=1}^{|\mathcal{C}|} 2\log(p/q) \sigma^2
\end{equation}
to find an optimal $\mathcal{C}$. The crucial part here is that the
coefficient estimate $\hat{\bbeta}_{\mathcal{C}}$ being used in
(\ref{eqn:modric}) is the least squares estimate of the true $\bbeta$
obtained by fitting $\y$ on
$\X_{\mathcal{C}} = (\x_j)_{j\in\mathcal{C}}$, and not the Lasso
estimate $\hat\bbeta_{l_1}$ provided by the algorithm. We also use
this least squares estimate in out-of-sample calculations.

\begin{figure}[htbp]
\begin{center}
\centerline{\includegraphics[width=.8\textwidth]{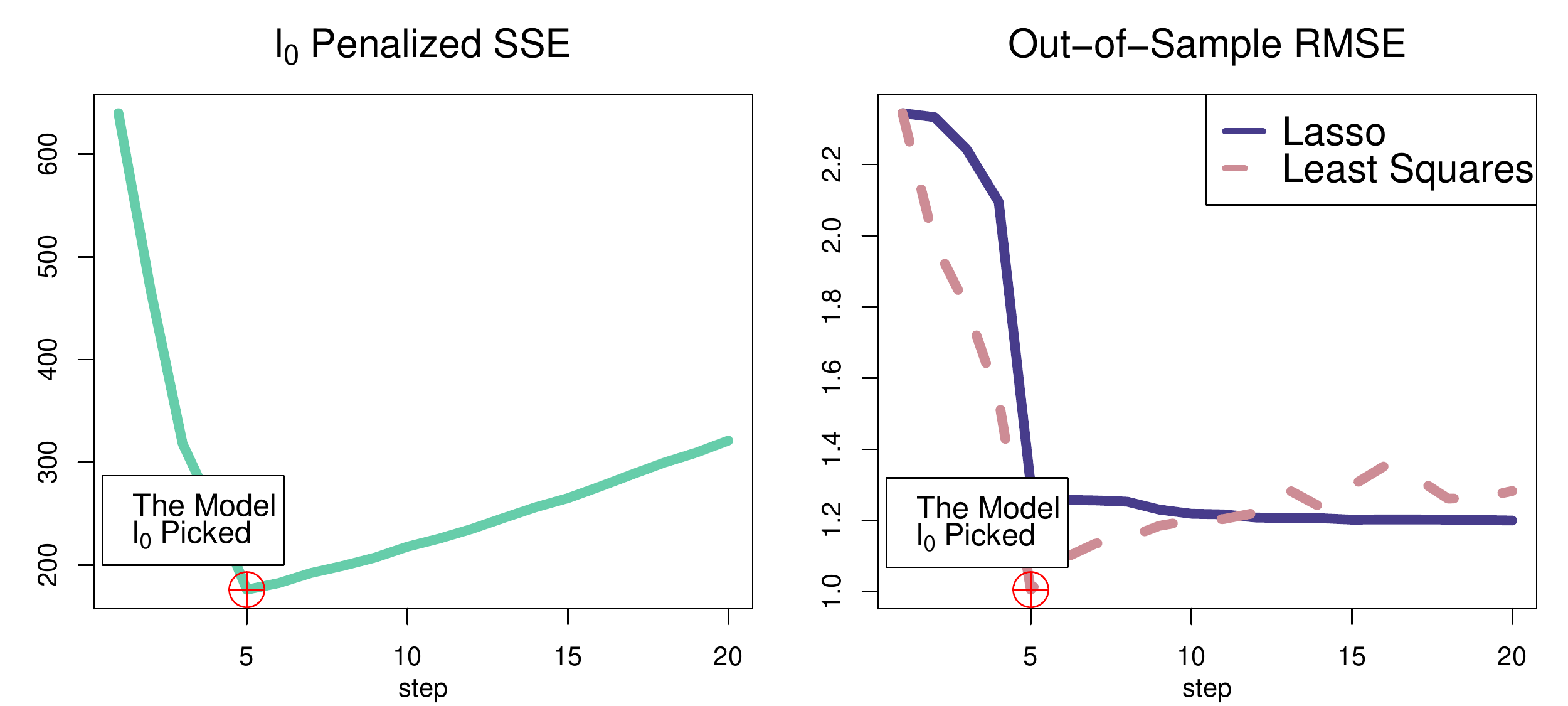}}
\vskip -.1in
\caption{{\footnotesize $l_0$ penalties help finding the best model
    (independent predictors case). $\mathbf{y}$ is simulated from one
    thousand features, only four of which have nonzero contributions,
    plus an $N(0,1)$ error. Both the training set and the test set
    have sizes $n=100$. Each step in the LARS algorithm gives a set of
    features with nonzero coefficient estimates. We compute the least
    squares (LS) estimates on this subset and the modified RIC
    criterion (\ref{eqn:modric}) on the training set. We also compare
    the out-of-sample root mean squared errors using the LS estimates
    and the Lasso estimates on this LARS path. The features are
    independently generated. The model that minimizes the $l_0$
    penalized error has exactly four variables in it. It also
    outperforms any of the $l_1$ models out-of-sample on this data
    set.}}
\label{fig:fittingls_ind}
\end{center}
\vskip -0.1in
\end{figure} 

\begin{figure}[htbp]
\begin{center}
\centerline{\includegraphics[width=.8\textwidth]{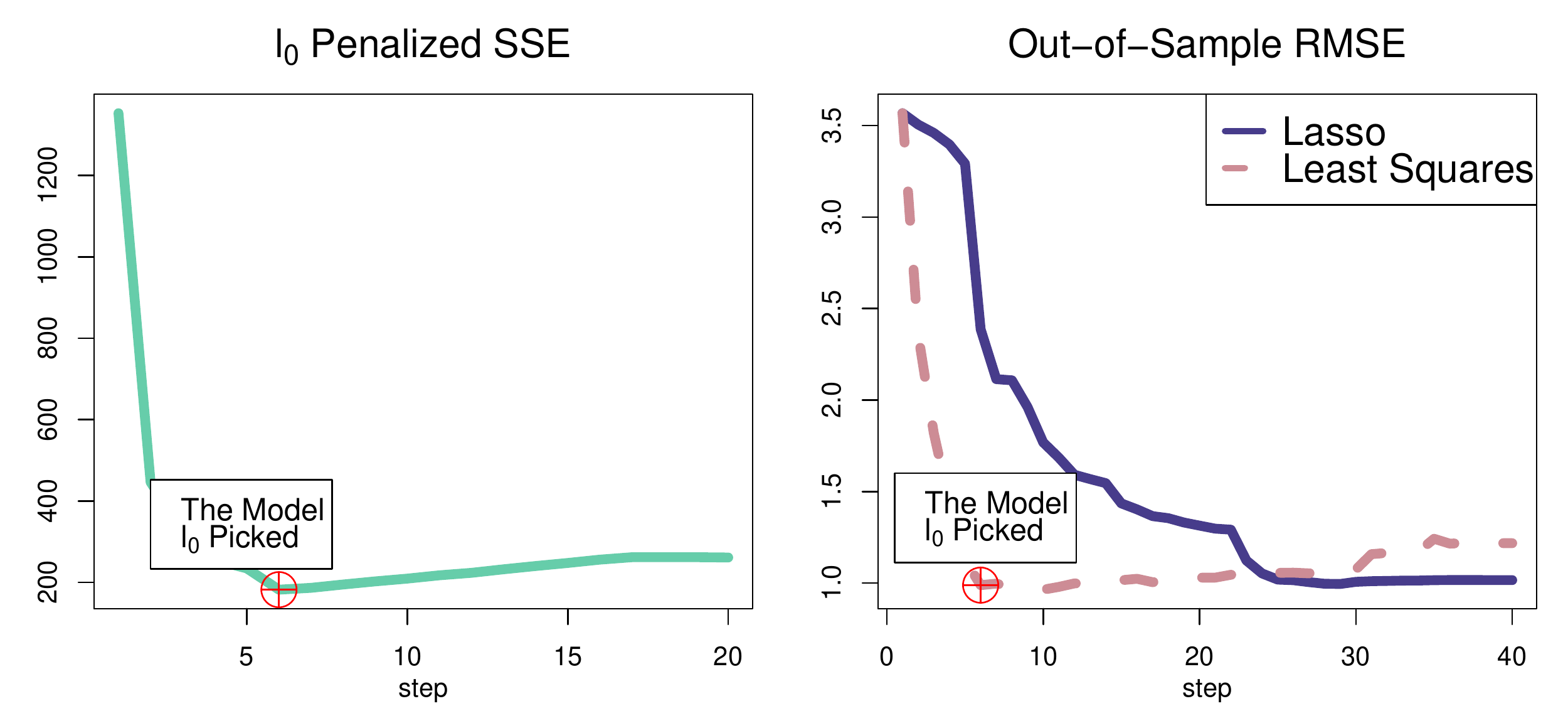}}
\vskip -.1in
\caption{{\footnotesize $l_0$ penalties help finding the best model
    (correlated predictors case). The setup is exactly the same as in
    Figure \ref{fig:fittingls_ind} except that each pair of features
    has a correlation $\rho=0.64$. In this case, the optimal model
    under the modified RIC criterion has a slightly better RMSE than
    the best $l_1$ model. The Lasso out-of-sample RMSE is typically
    minimized when the model has included more than 50 features.}}
\label{fig:fittingls_cor}
\end{center}
\vskip -0.1in
\end{figure}

We compare two cases: the $\mathbf{x}_j$'s are generated independently of
each other, meaning that $X'X$ is diagonal, and the $\mathbf{x}_j$'s are
generated with a pairwise correlation $\rho=0.64$. As shown in Figure
\ref{fig:fittingls_ind}, in the independent feature case, the model
picked by the modified RIC criterion always outperforms any Lasso
model on the test set. In the case with correlated predictors (Figure
\ref{fig:fittingls_cor}),
there is little difference between the out-of-sample accuracies of the
$l_0$-picked model and the best
Lasso model in this case, but Lasso adds around 50 more spurious variables.

Thus, by combining the computational efficiency of an $l_1$ algorithm
and the sparsity guaranteed by the $l_0$ penalization, we can easily
select an accurate model without cross validation.



\subsubsection{\torp{$l_0$}{l0} and NP-hardness}
\label{sec:np}

The $l_0$ problem is NP-hard and hence, at least in theory, intractable. (In
practice, of course, people often use approximate solutions to problems that in
the worst case can be NP-hard.)  One of the attractions of $l_1$-regularization
is that it is convex, hence solvable in polynomial time. In this section, we
compare how the two approaches fare on a known NP-hard regression problem. We
start with a simple constructive proof that the risk ratio for $l_1$ to $l_0$
can be arbitrarily bad. Construct data as follows. Pick a large number of
independent features $z_j$. Construct new features $x_1 = z_1 + \epsilon z_2$
and $x_2 = z_1 - \epsilon z_2$ and. Let $y = (z_1 + z_2)/2$ plus noise. Then the
correct model is $y = x_2/ \epsilon$. Include the rest of the features  $z_j, j
> 2$ as spurious features.

In \citet{Nata95} the known NP hard problem of ``the exact cover of 3-sets''
was reduced to the best subset selection problem as below:
$\mathbf{y}=\mathbf{1}_n$, $X$ is an $n\times p$ binary matrix with each column
having three nonzero elements: $\|\mathbf{x}_i\|_0=3$, $\bbeta$ is a $p\times 1$
vector, $\veps>0$ and we want to solve
\begin{IEEEeqnarray}{r'c'l}
\label{eqn:np}
\min_{\bbeta} \|\bbeta\|_0, & \text{s.t.} & \| \y - \X\bbeta \|_2<\veps.
\end{IEEEeqnarray}
Note that if there {\it is} a solution to this problem, the number of features
being chosen should be $n/3$. 

We then ask which method comes closer to solving this problem: a
greedy approximation to the $l_0$ problem or an exact solution to the
$l_1$ problem. To this end, we applied Lasso and forward stepwise
regression on various $n$'s. For small $n$'s, we took full collections
of the three subsets, i.e., $p$ equals $n$ choose 3; for larger $n$'s,
we took $p=10\cdot n$. Table \ref{table:np_exact} and
\ref{table:np_more} list the number of subsets included in the
model. Forward stepwise regression always finds fewer subsets, and
hence a better solution, than Lasso.

\begin{table}[h]
\begin{center}
\begin{tabular}{c|cccccccc}
{\bf Method} & $n=9$ & $n=12$ & $n=15$ & $n=18$ & $n=21$ & $n=24$ & $n=27$ &
$n=30$\\
\hline
\hline
{\bf Lasso}         & 6 & 10& 11 &17 & 19 & 21 & 22 &29\\
{\bf Stepwise}    & 3 & 4 & 5 &6 & 7 & 8 & 9 & 10\\
\end{tabular}
\caption{{\footnotesize The number of subsets chosen by Lasso and by
    forward stepwise regression with $\veps=1/4$. All 3-subsets
    were considered, i.e., $p=\binom{n}{3}$. Forward stepwise
    regression always has the fewest possible number of subsets,
    namely, $n/3$.}}
\label{table:np_exact}
\end{center}
\end{table}

\begin{table}[h]
\begin{center}
\begin{tabular}{c|ccccc}
{\bf Method} 	&$n=99$		&$n=240$	&$n=540$	&$n=990$
&$n=1500$\\
\hline
\hline
{\bf Lasso}	&93		&219	&504	&812	&1372\\
			&($2\times10^{-23}$)	&($9\times10^{-23}$)
&($9\times10^{-15}$)	&($6\times10^{-20}$)	& ($2\times10^{-20}$)\\
{\bf Stepwise}	&40		& 96		&223 	&364	&595\\
			&($1\times 10^{-28}$)	&($6\times10^{-27}$) 
&($3\times10^{-26}$) 	& ($6\times10^{-25}$) 	& ($1\times10^{-25}$)\\
\end{tabular}
\caption{{\footnotesize The number of subsets chosen by Lasso and by
    forward stepwise regression with $\veps=1/4$. $p=10\cdot n$
    3-subsets were randomly chosen to be the predictors. Forward
    stepwise regression always chooses a sparser solution in the sense
    that it chooses fewer number of subsets. Numbers in parentheses
    are the sum of squared errors when the algorithms terminated.}}
\label{table:np_more}
\end{center}
\end{table}

All of our experiments on both synthetic and real data sets show that
greedy search algorithms, such as stepwise regression, aimed at
minimizing $l_0$-regularized error provide sparser results. This is
because $l_0$ penalizes the sparsity directly, while $l_1$ does
not. It is easy to construct an example where $l_1$ will pick a
solution with a smaller $l_1$ norm but with a less sparse solution
\citep{CanWB07}.

\section{Conclusion}

In many statistical contexts, the $l_0$ regularization criterion is superior to
that of $l_1$ regularization; $l_0$ generally provides a more accurate solution
and controls the false discovery rate better.  $l_1$ can give arbitrarily worse
predictive accuracy than $l_0$, since $l_1$ regularization tends to shrink
coefficients too much to include many spurious features. Computationally, $l_1$
appears to be more attractive; convex programming makes the computation feasible
and efficient. In practice, however, approximate solutions to the $l_0$ problem
are often better than than exact solutions to the $l_1$ problem. The best
properties of the two methods can be combined.  Superior results were obtained
by using convex optimization of the $l_1$ problem to generate a set of candidate
models (the regularization path generated by LARS), and then selecting the best
model by minimizing the $l_0$-penalized training error.

\appendixpage
\begin{appendix}
\section{Risk Ratio Proofs}
\label{app:riskpf}

We will drop the $\gamma$'s when the situation is clear, and denote
$\hat\beta_{l_0}(\gamma_0)$ as $\hat\beta_{l_0}$ and
$\hat\beta_{l_1}(\gamma_1)$ as $\hat\beta_{l_1}$ for simplicity. 

Without loss of generality, we assume $\X'\X=I$ and $\sigma=1$. The $l_0$ risk
can be written as

\begin{IEEEeqnarray*}{rCl}
R(\bbeta,\hat\bbeta_{l_0}) & = & \E_{\bbeta} \|\X\bbeta-\X\hat\bbeta\|^2
    = \E_{\bbeta} \sum_{i=1}^p \|\x_i\|^2(\beta_i-\hat\beta_i)^2 \\
  & = & \E_{\bbeta} \sum_{i=1}^p\left(
    \left(\frac{\x_i'\bveps}{\|\x_i\|}\right)^2 I_{\{|\hat\beta_i| >
    \gamma\}}+ (\|\x_i\|\beta_i)^2 I_{\{|\hat\beta_i| \leq
    \gamma\}}\right) \yesnumber \\ 
  & = & \sum_{i=1}^p\left\{ \sigma^2\E_\beta\left[ 
    Z_i^2I_{\{|\beta_i+\sigma Z_i|>\gamma\}}\right]
    +(\|\x_i\|\beta_i)^2P(|\beta_i+\sigma Z_i|\leq\gamma)\right\},
\end{IEEEeqnarray*}
where $Z_i=\x_i'\bveps/\sigma\|\x_i\|\sim N(0,1)$, $i=1,\ldots,p$.

Similarly, the $l_1$ risk can be written as 
 \begin{IEEEeqnarray*}{rCl}
R(\bbeta,\hat\bbeta_{l_1}) & = & \E_{\bbeta}\sum_{i=1}^p \Bigg( 
    \left(\frac{\x_i'\bveps}{\|\x_i\|} - \tilde\gamma\right)^2 
    I_{\{\hat\beta _i > \tilde\gamma \}} +
    \left(\frac{\x_i'\bveps}{\|\x_i\|}+\tilde\gamma\right)^2
    I_{\{\hat\beta _i <- \tilde\gamma\}} \yesnumber \\ 
  & & + (\|\x_i\|\beta_i)^2 I_{\{|\hat\beta_i| \leq \tilde\gamma\}}
    \Bigg)\\
  & = & \sum_{i=1}^p \left\{ \E_\beta \left[ (\sigma Z_i - \tilde\gamma)^2
    I_{\{\beta_i+\sigma Z_i > \tilde\gamma\}}+(\sigma Z_i + \tilde\gamma)^2
    I_{\{\beta_i+\sigma Z_i < -\tilde\gamma\}} \right] \right.\\ 
  & & \left. + (\|\x_i\|\beta_i)^2 P(|\beta_i+\sigma Z_i|\leq\gamma)\right\}
\end{IEEEeqnarray*}

Specifically, we consider the case when $p=1$. Let
$\Phi(z)=P(Z\leq z)$ and $\tilde{\Phi}(z)=P(Z>z)$ be the lower and
upper tail probabilities of a standard normal distribution and the two
risk functions can be explicitly written as

\begin{IEEEeqnarray}{rCl}
R(\beta,\hat{\beta}_{l_0}) & = & \int_{\gamma_0-\beta}^{\infty} z^2 \phi(z)\ud z
    + \int_{-\infty}^{-\gamma_0-\beta} z^2 \phi(z)\ud z + \beta^2\left %
    [\Phi(\gamma_0-\beta) - \tilde\Phi (\gamma_0+\beta) \right] \nonumber\\ 
  & = & (\gamma_0-\beta) \phi(\gamma_0-\beta) +
    (\gamma_0+\beta)\phi(\gamma_0+\beta) \label{eqn:l0risk}\\ 
  & & + \Phi(\beta-\gamma_0) + \beta^2 \Phi(\gamma_0-\beta)
    + (1-\beta^2) \tilde\Phi (\gamma_0+\beta), \nonumber\\
R(\beta,\hat\beta_{l_1}) & = & \int_{\gamma_1-\beta}^{\infty} (z-\gamma_1)^2
    \phi(z) \ud z + \int_{-\infty}^{ -\gamma_1-\beta} (z+\gamma_1)^2 \phi(z) 
    \ud z \nonumber \\ 
  & & + \beta^2\left[ \Phi(\gamma_1-\beta) - \tilde\Phi(\gamma_1+\beta)
    \right] \nonumber \\ 
  & = & (-\gamma_1-\beta) \phi(\gamma_1-\beta) + (\beta-\gamma_1)
    \phi(\gamma_1+\beta) \label{eqn:l1risk}\\ 
  & & + (\gamma_1^2+1)
    \Phi(\beta-\gamma_1) + \beta^2\Phi(\gamma_1-\beta)   + (\gamma_1^2+1 
    -\beta^2) \tilde\Phi(\gamma_1+\beta)\nonumber 
\end{IEEEeqnarray}

\vspace{.1in}

We list a few Gaussian tail bounds here that we will use in the proofs
later. Detailed discussion can be found in related articles \citep{Feller68,
DonJ94,FosterG94,Abram+06}.

\begin{lemma}
\label{lem:normal}
For any $z>o$, 
\begin{enumerate}
  \item $\phi(z)(z^{-1}-z^{-3})\leq \tilde{\Phi}(z)\leq \phi(z)z^{-1}$;
  \item $\tilde{\Phi}(z)\leq e^{-z^2}/2$.
  \item $\phi(z)(x^{-1}-x^{-3}+(1\cdot 3)\cdot x^{-5}-(1\cdot 3\cdot
    5)\cdot x^{-7}+\cdots+(-1)^{k}\cdot(2k-1)!!\cdot x^{-2k-1}$ overestimates
    $\tilde{\Phi}(z)$ if $k$ is even, and underestimates $\tilde{\Phi}(z)$ if
    $k$ is odd.
  \item $\phi(z)(z^{-1}-z^{-3}+(1\cdot 3)\cdot z^{-5}-(1\cdot 3\cdot
    5)\cdot z^{-7}+\cdots+(-1)^{k}\cdot(2k-1)!!\cdot z^{-2k-1}$ overestimates
    $\tilde{\Phi}(z)$ if $k$ is even, and underestimates $\tilde{\Phi}(z)$ if
    $k$ is odd. \mynotesS{check if x or z is correct}
\end{enumerate}
\end{lemma}

\vspace{.1in}

\begin{lemma}
For large enough $\gamma_0>0$, 
\begin{IEEEeqnarray}{rCl}
\label{eqn:lowerbound}
\inf_{\gamma_1} \sup_{\beta} \frac{R(\beta,
\hat\beta_{l_1})}{R(\beta,\hat\beta_{l_0})} & > & \gamma_0.
\end{IEEEeqnarray}
\end{lemma}

\begin{proof}
It suffices to show that for any fixed $\gamma_0$ and any $\gamma_1$
\begin{IEEEeqnarray*}{rCl}
\sup_{\beta} \frac{R(\beta,
\hat\beta_{l_1})}{R(\beta,\hat\beta_{l_0})} & > & \gamma_0.
\end{IEEEeqnarray*}

Suppose $\gamma_1\geq\gamma_0/\sqrt{2}$, let $\beta_n = (n+1)\gamma_0$, then 
\begin{IEEEeqnarray*}{rCcCcCl}
\|\hat\beta_{l_1} - \hat\beta_{LS}\|_2^2 & > & 
\|\hat\beta_{l_1} - \hat\beta_{LS}\|_2^2 I_{\{\hat\beta_{LS}>\gamma_1\}} & = &
\gamma_1^2 I_{\{\hat\beta_{LS} > \gamma_1 \}} & \geq &
\frac{\gamma_0^2}{2} I_{\{\hat\beta_{LS}>\gamma_1\}}.
\end{IEEEeqnarray*}
Hence,
\begin{IEEEeqnarray*}{rCl}
\E\|\hat\beta_{l_1} - \hat\beta_{LS}\|_2^2 & > &
\frac{\gamma_0^2}{2} P(\hat\beta_{LS}>\gamma_1),
\end{IEEEeqnarray*}
Thus,
\begin{IEEEeqnarray*}{rCl}
\E\|\hat\beta_{l_1} - \beta_n\|_2^2 & \geq & E\|\hat\beta_{l_1} -
    \hat\beta_{LS}\|_2^2 - E\|\hat\beta_{LS} - \beta_n\|_2^2 >
    \frac{\gamma_0^2}{2} P(\hat\beta_{LS} > \gamma_1)-1\\
& = & \left (\frac{\gamma_0^2}{2} - 1\right) P(\hat\beta_{LS} > \gamma_1)
    - P(\hat{\beta}_{LS} \leq \gamma_1)\\
& > & \gamma_0\Phi((n+1) \gamma_0 - \gamma_1) - \Phi(\gamma_1 - (n+1)\gamma_0), 
\end{IEEEeqnarray*}
for large enough $\gamma_0$ and $Z \sim N(0,1)$.

On the other hand,
\begin{IEEEeqnarray*}{rCl}
\E\|\hat\beta_{l_0} - \beta_n\|_2^2
& = & -n\gamma_0 \phi(n\gamma_0) + (n+2)\gamma_0 \phi \left( (n+2)\gamma_0
  \right) + \Phi(n\gamma_0) \\
& & +(n+1)^2 \gamma_0^2 \tilde\Phi(n\gamma_0) + (1-(n+1)^2 \gamma_0^2)
  \tilde\Phi((n+2)\gamma_0)\\
& \leq & 1+\left (-n\gamma_0 - \frac{1}{n \gamma_0} + \frac{(n+1)^2
  \gamma_0}{n} \right) \phi(n\gamma_0)\\
& & +\left ((n+2)\gamma_0 + \frac{1-(n+1)^2 \gamma_0^2}{(n+2) \gamma_0}
  \right) \phi((n+2) \gamma_0)\\
& \leq & 1 + \left( 2+\frac{1}{n} + 2e^{-2(n+1) \gamma_0^2} \right)
  \gamma_0 \phi(n\gamma_0).
\end{IEEEeqnarray*}

Hence,
\begin{IEEEeqnarray*}{rCl}
\frac{R(\beta_n, \hat\beta_{l_1})}{R(\beta_n, \hat\beta_{l_0})} & \geq & 
\frac{\gamma_0 \Phi((n+1)\gamma_0 - \gamma_1) - \Phi(\gamma_1 -
  (n+1)\gamma_0)}{1 + \left (2 + n^{-1} + 2e^{-2(n+1)\gamma_0^2} \right)
  \gamma_0 \phi(n\gamma_0)}
\end{IEEEeqnarray*}

Let $n \rightarrow \infty$, then 
\begin{IEEEeqnarray}{rCcCl}
\sup_\beta\frac{R(\beta, \hat\beta_{l_1})}{R(\beta, \hat\beta_{l_0})}
& \geq & \lim_{n \rightarrow \infty} \frac{R(\beta_n, \hat\beta_{l_1})}
  {R(\beta_n, \hat\beta_{l_0})}
& \geq & \gamma_0.
\end{IEEEeqnarray}

For those $0 \leq \gamma_1 < \gamma_0 / \sqrt{2}$, we consider $\beta = 0$ and
denote 
\begin{IEEEeqnarray}{rCcCl}
R_0(\gamma_0) & = & R(0, \hat\beta_{l_0}(\gamma_0)) & = &
  2\gamma_0\phi(\gamma_0) +  2\tilde \Phi(\gamma_0)\\
R_1(\gamma_1) & = & R(0, \hat\beta_{l_1}(\gamma_1)) & = & -2\gamma_1
  \phi(\gamma_1) + 2(\gamma_1^2+1) \tilde\Phi(\gamma_1).
\end{IEEEeqnarray}

Consider $0 < c \leq \gamma_1 < \gamma_0 / \sqrt{2}$. By Lemma \ref{lem:normal},
for any $z \geq c$, $\tilde\Phi(z) - \phi(z)(1/z-1/z^3+1/z^5) \geq 0$. We have
$\phi(\gamma_1) \geq \phi(\gamma_0) e^{\gamma_0^2/4}$, and
\begin{IEEEeqnarray*}{rCl}
R_1(\gamma_1) & \geq & -2\gamma_1 \phi(\gamma_1) + 2(\gamma_1^2+1)
  \left(\gamma_1^{-1} -\gamma_1^{-3} + \gamma_1^{-5} \right) \phi(\gamma_1)\\
& \geq & 2^{7/2} \gamma_0^{-5} \phi(\gamma_0) e^{ \gamma_0^2/4}\\
R_0(\gamma_0) & \leq & 2\left(\gamma_0 + \gamma_0^{-1} \right)\phi(\gamma_0).
\end{IEEEeqnarray*}
Thus for large enough $\gamma_0$
\mynotesL{$0<\gamma_0<1.265$ also works}
\begin{IEEEeqnarray*}{rCcCl}
\frac{R_1(\gamma_1)}{R_0(\gamma_0)} & \geq & \frac{2^{5/2} e^{\gamma_0^2/4}}{
  \gamma_0^6 + \gamma_0^4} & > & \gamma_0.
\end{IEEEeqnarray*}

Lastly, since
\begin{IEEEeqnarray*}{rCcCl}
\frac{d}{d\gamma_1} R_1(\gamma_1) & = &
-4\phi(\gamma_1)+4\gamma_1\tilde\Phi (\gamma_1) & < & 0,
\end{IEEEeqnarray*}
for any $0 \leq \gamma_1 \leq c$, we have
\begin{IEEEeqnarray*}{rCcCl+x*}
\frac{R_1(\gamma_1)}{R_0(\gamma_0)} & \geq &
\frac{R_1(c)}{R_0(\gamma_0)} & > & \gamma_0. & \qedhere
\end{IEEEeqnarray*}
\end{proof}

\vspace{.1in}

\begin{proof}[Proof of Theorem \ref{thm:low}] 
Let $C_1 = \min_{\gamma_0>0} \left\{\frac{2^{5/2} e^{\gamma_0^2/4}}{\gamma_0^6 +
\gamma_0^4} - \gamma_0 \right\} < 5.161$, which occurs at $\gamma_0 = 5.71$. 
\end{proof}

%

\vspace{.1in}

\begin{lemma}
\label{lem:upperbound}
There exists an $M > 0$ and a constant $C > 0$, such that for all
$\gamma_1 > M$,
\begin{IEEEeqnarray}{rCl}
\label{eqn:upperbound}
\inf_{\gamma_0}\sup_\beta \frac{R(\beta,  \hat\beta_{l_0})}{R(\beta,
\hat\beta_{l_1})} & \leq &
  1 + C\gamma_1^{-1}
\end{IEEEeqnarray}
\end{lemma}

It suffices to show that for all $\beta \geq 0$ and particular values of
$\gamma_0$, we have
\begin{IEEEeqnarray}{rCl}
\label{eqn:allupperbound}
\frac{R(\beta, \hat\beta_{l_0})}{R(\beta, \hat\beta_{l_1})} & \leq & 
  1+C \gamma_1^{-1}.
\end{IEEEeqnarray}

The proof is done by generating bounds for the risks at various $\beta$'s.
Before giving these proofs, we need to relate $R(\beta, \hat\beta_{l_0})$ to
$R(\beta, \hat\beta_{l_1})$.
\begin{IEEEeqnarray*}{rCl}
R(\beta, \hat\beta_{l_0}) & = & (\gamma_0 - \beta)\phi(\gamma_0 - \beta) +
  (\gamma_0 + \beta)\phi(\gamma_0  + \beta) \\
& & +\Phi(-\gamma_0+\beta)+\beta^2\Phi(\gamma_0-\beta)+(1-\beta^2)\tilde{\Phi}
  (\gamma_0+\beta)\\
& = & (\gamma_1 + \Delta \gamma - \beta) \phi(\gamma_1-\beta) + (\gamma_1 +
  \Delta\gamma - \beta)\left. \pd{}{\gamma} \phi(\gamma-\beta)
  \right|_{\gamma_1} \Delta\gamma\\
& & +(\gamma_1 + \Delta\gamma + \beta) \phi(\gamma_1 + \beta) + (\gamma_1 +
  \Delta\gamma + \beta)\left. \pd{}{\gamma} \phi(\gamma+\beta)
  \right|_{\gamma_1} \Delta\gamma\\
& & +\Phi(-\gamma_1+\beta) + \left. \pd{}{\gamma} \Phi(-\gamma+\beta)
  \right|_{\gamma_1} \Delta\gamma + \beta^2 \Phi(\gamma_1-\beta) + \beta^2
  \left. \pd{}{\gamma} \Phi(\gamma-\beta) \right|_{\gamma_1} \Delta\gamma\\
& & +(1-\beta^2) \tilde\Phi(\gamma_1+\beta) + (1-\beta^2) \left.
  \pd{}{\gamma} \tilde\Phi(\gamma+\beta) \right|_{\gamma_1}
  \Delta\gamma + \gamma_1 e^{-\gamma_1^2/2} o(\Delta\gamma)\\
& = & (\gamma_1-\beta) \phi(\gamma_1-\beta) + (\gamma_1+\beta)
  \phi(\gamma_1+\beta) + \Phi(-\gamma_1+\beta) + \beta^2 \Phi(\gamma_1-\beta)\\
& & +(1-\beta^2) \tilde\Phi (\gamma_1+\beta) - (\gamma_1^2 - 2\beta \gamma_1)
  \phi(\gamma_1-\beta) \Delta\gamma\\
& & -(\gamma_1^2 + 2\beta\gamma_1) \phi(\gamma_1+\beta) \Delta\gamma +
  \gamma_1 e^{ -\gamma_1^2} o(\Delta\gamma)\\
& = & R(\beta,\hat{\beta}_{l_1}) + 2\gamma_1 \phi(\gamma_1-\beta) + 2\gamma_1
  \phi(\gamma_1+\beta) - \gamma_1^2 \Phi(- \gamma_1+\beta) - \gamma_1^2
  \tilde\Phi (\gamma_1+\beta)\\
& & -(\gamma_1^2 - 2\beta\gamma_1) \phi(\gamma_1-\beta) \Delta\gamma
  -(\gamma_1^2 + 2\beta\gamma_1) \phi(\gamma_1+\beta) \Delta\gamma +
  \gamma_1e^{-\gamma_1^2/2} o(\Delta\gamma)
\end{IEEEeqnarray*}

We can now provide separate proofs for $\beta$ within the following regions:
\begin{enumerate}
 \item $0 \leq \beta \leq \gamma_1 - \sqrt{\log(\gamma_1/2)}$
 \item $\gamma_1 - \sqrt{\log(\gamma_1/2)} < \beta \leq
  \gamma_1+\sqrt{2\log(\gamma_1)}$
 \item $\gamma_1+\sqrt{2\log(\gamma_1)} < \beta$
\end{enumerate}

\begin{proof}[Proof for case 1, $0 \leq \beta \leq \gamma_1 -
\sqrt{\log(\gamma_1/2)}$]
 Use the trivial estimator $\hat \beta_{l_0} = 0$, ie. set $\gamma_0 = \infty$.
Then,
\begin{IEEEeqnarray*}{rCl}
 R(\beta, \hat\beta_0) & = & \beta^2\\
 && \text{and}\\
R(\beta, \hat\beta_1) & = & \E_\beta \left[ (Z - \gamma_1)^2
  I_{\{\beta + Z > \gamma_1\}}+(Z + \gamma_1)^2
  I_{\{\beta + Z < -\gamma_1\}} \right] + (\|\x\|\beta)^2
  \Prob(|\beta + Z| \leq \gamma_1)\\
& > & \beta^2 \Prob(-\gamma_1 - \beta \leq Z \leq \gamma_1 - \beta)\\
& > & \beta^2\Prob\left( -2\gamma_1 + \sqrt{\log(\gamma_1/2)} \leq Z \leq
  \sqrt{\log(\gamma_1/2)} \right)\\
& = & \beta^2 \left(1 - \Phi \left(-2\gamma_1 + \sqrt{\log(\gamma_1/2)} \right)
-  \tilde\Phi \left(\sqrt{\log(\gamma_1/2)} \right) \right)\\
& = & \beta^2 \left(1 - \tilde\Phi \left(2\gamma_1 - \sqrt{\log(\gamma_1/2)}
  \right) - \tilde\Phi \left(\sqrt{\log(\gamma_1/2)} \right) \right)\\
& = & \beta^2 \left(1 - 2\tilde\Phi \left(\sqrt{\log(\gamma_1/2)} \right)
  \right)\\
& > & \beta^2 \left(1 - \exp\left\{-\left( \sqrt{\log(\gamma_1/2)}
  \right)^2\right\} \right)\\
& = & \beta^2 \left(1 - \frac{2}{\gamma_1}\right)\\
&& \text{therefore}\\
\frac{R(\beta, \hat\beta_{l_0})}{R(\beta, \hat\beta_{l_1})} & \leq &
  \frac{\beta^2}{\beta^2 \left(1 - \frac{2}{\gamma_1}\right)}\\
& = & \frac{\gamma_1}{\gamma_1-2}\\
& = & 1 + \frac{2}{\gamma_1-2}\\
& = & 1 + o(\gamma_1^{-1}) \text{, if $\gamma_1 >2$.}
\end{IEEEeqnarray*}
We have implicitly defined $\frac{0}{0} = 1$, which can be justified in
this case using a limit argument.
\end{proof}

\begin{proof}[Proof for case 2, $\gamma_1 - \sqrt{\log(\gamma_1/2)} < \beta
\leq \gamma_1+\sqrt{2\log(\gamma_1)}$]
Recall that
\begin{IEEEeqnarray*}{rCl}
R(\beta, \hat\beta_0) - R(\beta, \hat\beta_1) & = & (2\gamma_1 + 2\Delta\gamma +
  2\beta\gamma_1\Delta\gamma - \gamma_1^2\Delta\gamma) \phi(\gamma_1-\beta) -
  \gamma_1^2 \Phi(-\gamma_1+\beta)\\
&& + (2\gamma_1 + 2\Delta\gamma - \gamma_1^2\Delta\gamma -
  2\beta\gamma_1\Delta\gamma) \phi(\gamma_1+\beta) - \gamma_1^2
  \tilde\Phi(\gamma_1+\beta) + o(\Delta\gamma)
\end{IEEEeqnarray*}
We want to replace $\phi(\gamma_1+\beta)$ by $\phi(\gamma_1-\beta)$. Need to
make sure that the term on $\phi(\gamma_1+\beta)$ is positive, so that this in
fact increases the difference. This holds for $\beta \leq \gamma_1 +
\sqrt{2\log\gamma_1}$ and $\Delta\gamma = \frac{1}{2\gamma_1}$,
\begin{IEEEeqnarray*}{rCl}
2\gamma_1 + 2\Delta\gamma - \gamma_1^2\Delta\gamma - 2\beta\gamma_1
  \Delta\gamma & = & 2\gamma_1 + 2\frac{1}{2\gamma_1} -
\gamma_1^2\frac{1}{2\gamma_1} - 2\beta\gamma_1 \frac{1}{2\gamma_1}\\
& = & 2\gamma_1 + \frac{1}{\gamma_1} -
\frac{\gamma_1}{2} - \beta\\
& > & 2\gamma_1 + \frac{1}{\gamma_1} -
\frac{\gamma_1}{2} - \gamma_1-\sqrt{2\log(\gamma_1)}\\
& = & \frac{\gamma_1}{2} + \frac{1}{\gamma_1} - \sqrt{2\log(\gamma_1)}\\
& > & 0 \text{, if $\gamma_1 \geq 1$.}
\end{IEEEeqnarray*}

\ifx\FormatStyle\InHouse
The above inequality must hold for the proofs using the taylor expansion of
$R(\beta,\hat\beta_0)$ so that the  $\beta$ terms can be canceled (at least in
this proof). If the bound on $\beta$ grows at a rate less than $\gamma_1$, this
puts a restriction on $\Delta\gamma$. Namely, $\Delta\gamma <
\frac{2}{3\gamma_1}$ in order for the inequality to hold. (at least for bounds
of the form$\beta < \gamma_1 + \kappa$ as shown here.) 
\fi

Making the replacement above,
\begin{IEEEeqnarray*}{rCl}
R(\beta, \hat\beta_0) - R(\beta, \hat\beta_1) & < & (4\gamma_1 + 4\Delta\gamma -
  2\gamma_1^2\Delta\gamma) \phi(\gamma_1-\beta) + o(\gamma_1^{-1})\\
& = & (4\gamma_1 + 4\frac{1}{2\gamma_1} - 2\gamma_1^2\frac{1}{2\gamma_1})
  \phi(\gamma_1-\beta) + o(\gamma_1^{-1})\\
& < & \left( 3\gamma_1 + \frac{2}{\gamma_1}\right)
  \phi(\gamma_1-\beta) + o(\gamma_1^{-1})
\end{IEEEeqnarray*}

Next, we need a lower bound $R_{l_1}$.
\begin{IEEEeqnarray*}{rCl}
R(\beta,\hat\beta_{l_1}) & = & (-\gamma_1-\beta) \phi(\gamma_1-\beta) +
  (\beta-\gamma_1) \phi(\gamma_1+\beta)\\ 
&& + (\gamma_1^2+1) \Phi(\beta-\gamma_1) + \beta^2\Phi(\gamma_1-\beta)   +
  (\gamma_1^2+1 -\beta^2) \tilde\Phi(\gamma_1+\beta)\\
& = & (-\gamma_1-\beta) \phi(\gamma_1-\beta) +
  (\beta-\gamma_1) \phi(\gamma_1+\beta)\\
&& + (\gamma_1^2+1) \Tilde\Phi(\gamma_1-\beta) +
  \beta^2\Phi(\gamma_1-\beta)  + (\gamma_1^2+1 -\beta^2) \Phi(-\gamma_1-\beta)\\
& = & (-\gamma_1-\beta) \phi(\gamma_1-\beta) + (\beta-\gamma_1)
  \phi(\gamma_1+\beta)\\
&& + (\gamma_1^2+1)(1 - \Phi(\gamma_1-\beta)) + \beta^2\Phi(\gamma_1-\beta) +
  (\gamma_1^2+1 -\beta^2) \Phi(-\gamma_1-\beta)\\
& = & (-\gamma_1-\beta) \phi(\gamma_1-\beta) + (\beta-\gamma_1)
  \phi(\gamma_1+\beta)\\
&&+ (\gamma_1^2+1) - (\gamma_1^2+1)\Phi(\gamma_1-\beta)) +
  \beta^2\Phi(\gamma_1-\beta) + (\gamma_1^2+1 -\beta^2) \Phi(-\gamma_1-\beta)\\
& = & (-\gamma_1-\beta) \phi(\gamma_1-\beta) + (\beta-\gamma_1)
  \phi(\gamma_1+\beta)\\
&&+ (\gamma_1^2+1) - (\gamma_1^2 + 1 - \beta^2)\Phi(\gamma_1-\beta)) +
  (\gamma_1^2+1 -\beta^2) \Phi(-\gamma_1-\beta)
\end{IEEEeqnarray*}

For replacement while keeping bounds, separate into two cases: 1) $\beta^2 \geq
\gamma_1^2+1$, and 2) $\beta^2 < \gamma_1^2+1$.
\begin{IEEEeqnarray*}{rCl}
R(\beta,\hat\beta_{l_1}) & = & (-\gamma_1-\beta) \phi(\gamma_1-\beta) +
  (\beta-\gamma_1) \phi(\gamma_1+\beta)\\
&&+ (\gamma_1^2+1) - (\gamma_1^2 + 1 - \beta^2)\Phi(\gamma_1-\beta)) +
  (\gamma_1^2+1 -\beta^2) \Phi(-\gamma_1-\beta)\\
&& \text{ Case 1}\\
& > & -\gamma_1 \left( \phi(\gamma_1-\beta) + \phi(\gamma_1+\beta) \right)
- \beta \left( \phi(\gamma_1-\beta) - \phi(\gamma_1+\beta) \right) \\
&& (\gamma_1^2+1) - (\gamma_1^2 + 1 - \beta^2)\Phi(-\gamma_1-\beta) +
(\gamma_1^2+1 -\beta^2) \Phi(-\gamma_1-\beta)\\ 
& > & -\gamma_1-\beta \phi(\gamma-\beta) + \gamma_1^2+1\\
& > & -\gamma_1 - \gamma_1/2 + \gamma_1^2+1\\
& = & \gamma_1^2 - \frac{3\gamma_1}{2} + 1\\
& > & 0 \text{, if $\gamma_1>1$.}\\
&& \text{ Case 2}\\
& = & -\gamma_1 \left( \phi(\gamma_1-\beta) + \phi(\gamma_1+\beta) \right)
- \beta \left( \phi(\gamma_1-\beta) - \phi(\gamma_1+\beta) \right) \\
&& (\gamma_1^2+1) - (\gamma_1^2 + 1 - \beta^2)\Phi(\gamma_1-\beta) +
(\gamma_1^2+1 -\beta^2) \Phi(-\gamma_1-\beta)\\ 
& > & -\gamma_1-\beta \phi(\gamma-\beta) + (\gamma_1^2+1) -
(\gamma_1^2 + 1 - \beta^2)\Phi(\gamma_1-\beta)\\
& > & -\frac{3\gamma_1}{2} + \gamma_1^2+1 - \gamma_1^2 - 1 + \beta^2\\
& > & -\frac{3\gamma_1}{2} + \left( \gamma_1 - \sqrt{\log(\gamma_1/2)}
  \right)^2\\
& > & 0 \text{, if $\gamma_1 \geq 2$.}
\end{IEEEeqnarray*}

Using worst case $\beta^2 < \gamma_1^2+1$, the above yields
\begin{IEEEeqnarray*}{rCl+x*}
\frac{R(\beta, \hat\beta_{l_0})}{R(\beta, \hat\beta_{l_1})} & \leq & 1 +
  \frac{\frac{3\gamma_1}{2} + \frac{1}{\gamma_1} + o(\gamma_1^{-1})}
  {-\frac{3\gamma_1}{2} + \left( \gamma_1 - \sqrt{\log(\gamma_1/2)}
  \right)^2}\\
& = & 1 + o(\gamma_1^{-1}) \text{, for $\gamma_1 \geq 2$.} & \qedhere
\end{IEEEeqnarray*}
\end{proof}

\begin{proof}[Proof for case 3, $\beta > \gamma_1+\sqrt{2\log(\gamma_1)}$.]
Let $\Delta\gamma = 0$,
\begin{IEEEeqnarray*}{rCl+x*}
R(\beta, \hat\beta_{l_0}) - R(\beta,\hat{\beta}_{l_1}) & = & 2\gamma_1
  \phi(\gamma_1-\beta) + 2\gamma_1 \phi(\gamma_1+\beta) - \gamma_1^2 \Phi(-
  \gamma_1+\beta) - \gamma_1^2 \tilde\Phi (\gamma_1+\beta)\\
& < & 4\gamma_1 \phi(\sqrt{2\log\gamma_1}) -
  \gamma_1^2\Phi(\sqrt{2\log\gamma_1})\\
& = & \frac{4}{\sqrt{2\pi}} - \gamma_1^2\Phi(\sqrt{2\log\gamma_1})\\
& < & 0 \text{, if $\gamma_1 \geq 1.415$.} & \qedhere
\end{IEEEeqnarray*}
\end{proof}

\begin{proof}[Proof of Lemma \ref{lem:upperbound}]
Using the above proofs, let M = 2 and C the constant suppressed in
$o(\gamma_1^{-1})$). 
\end{proof}
\mynotesH{is it possible to identify this constant?}

\begin{proof}[Proof of Theorem \ref{thm:up}]
For $\gamma_1 < M$ we know that there exists some $\epsilon > 0$ such that
$R(\beta, \hat\beta_{l_1}(\gamma_1)) \geq \epsilon$ for all $\beta$.  If we use
the trivial estimator $\gamma_0 = 0$, we know it has risk $1$. Hence, we can
pick $C_2 = \max(1 / \epsilon, C)$ where $C$ is from our lemma, then Theorem
\ref{thm:up} follows.
\end{proof}
\mynotesH{determine $\epsilon$}

\end{appendix}



\bibliography{/media/sf_Dropbox/Research/Bib_Stuff/Bibliography.bib}
\end{document}